\documentclass[12pt,reqno]{amsproc}

\title[Discrete Morse Theory for Computing Cellular Sheaf Cohomology]{Discrete Morse Theory for Computing\\Cellular Sheaf Cohomology}

\author{Justin Curry}
\address{Department of Mathematics,
        University of Pennsylvania,
        Philadelphia PA, USA}
\email{jcurry@math.upenn.edu}

\author{Robert Ghrist}
\address{Departments of Mathematics and Electrical/Systems Engineering,
        University of Pennsylvania,
        Philadelphia PA, USA}
\email{ghrist@math.upenn.edu}

\author{Vidit Nanda}
\address{Department of Mathematics,
        University of Pennsylvania,
        Philadelphia PA, USA}
\email{vnanda@sas.upenn.edu}

\usepackage{amsmath,mathrsfs}
\usepackage{amssymb,latexsym,nicefrac}
\usepackage[osf]{mathpazo}
\usepackage[euler-digits]{eulervm}
\usepackage{graphicx}
\usepackage[usenames,dvipsnames]{color}
\usepackage{xypic}
\usepackage[cmtip,arrow]{xy}
\usepackage[margin=0.9in]{geometry}

\usepackage{hyperref}
\hypersetup
{
    colorlinks,	%
    citecolor=Blue,%
    filecolor=Blue,%
    linkcolor=Blue,%
    urlcolor=Gray
}

\newtheorem{thm}{Theorem}[section]
\newtheorem*{thm*}{Theorem}
\newtheorem{lem}[thm]{Lemma}
\newtheorem{defn}[thm]{Definition}

\newtheorem{prop}[thm]{Proposition}

\newtheorem{rem}[thm]{Remark}

\newcommand{\setof}[1]{\left\{ {#1}\right\}}

\newcommand{\bR}{{\bf R}}

\newcommand{\N}{{\mathbb{N}}}
\newcommand{\R}{{\mathbb{R}}}

\newcommand{\Z}{{\mathbb{Z}}}

\newcommand{\cA}{{\mathcal A}}

\newcommand{\cF}{{\mathcal F}}

\newcommand{\cU}{{\mathcal U}}
\newcommand{\cV}{{\mathcal V}}

\newcommand{\cX}{{\mathcal X}}

\newcommand{\src}{{\mathsf s}}
\newcommand{\tgt}{{\mathsf t}}

\newcommand{\ow}{{\overline{w}}}

\newcommand{\Que}{{\sf{Que }}}

\def\mapright#1{\stackrel{#1}{\longrightarrow}}

\def\setof#1{\left\{{#1}\right\}}

\newcommand{\id}{\text{id}}
\newcommand{\red}{\star}

\newcommand{\hsp}{\hspace{-0.5em}}
\newcommand{\rank}{\hbox{\rm rank}\,}

\DeclareMathOperator{\img}{img}

\newcommand{\tab}{\hspace{15pt}}

\definecolor{shadecolor}{RGB}{210,210,250}

{\begin{snugshade}\begin{quote}}
{\hfill\end{quote}\end{snugshade}}

\newcommand{\style}[1]{{\bf{#1}}\index{{#1}}}    
\newcommand{\Sheaf}{\mathcal F}

\newcommand{\Leray}{\mathcal L}
\newcommand{\Cech}{\mathcal C}

\begin{document}

\begin{abstract}
Sheaves and sheaf cohomology are powerful tools in computational topology, greatly generalizing persistent homology. We develop an algorithm for simplifying the computation of cellular sheaf cohomology via (discrete) Morse-theoretic techniques. As a consequence, we derive efficient techniques for distributed computation of (ordinary) cohomology of a cell complex.
\end{abstract}

\maketitle

\section{Introduction} \label{sec:intro}

\subsection{Computational topology and sheaves}
It has recently become clear that computation of homology of spaces is of critical importance in several applied contexts. These include but are not limited to configuration spaces in robotics \cite{Farber,Farber:graph,G:singapore,CohenKod}, the global qualitative statistics of point-cloud data \cite{carlsson:09,pyramid,EdelsHar}, coverage problems in sensor networks \cite{DG:ijrr,DG:persistence}, circular coordinates for data sets \cite{DMV}, and Conley-type indices for dynamics \cite{kaczynski:mischaikow:mrozek:04,araietc,MM}. The Euler characteristic -- a numerical reduction of homology -- is even more ubiquitous, with applications ranging from Gaussian random fields \cite{Adler,AdlerTaylor} to data aggregation problems over networks \cite{BG2009,BG2010} and signal processing \cite{curry:ghrist:robinson:12}. Not coincidentally, development of applications of homological tools has proceeded symbiotically with the development of good algorithms for computational homology \cite{kaczynski:mischaikow:mrozek:04,EdelsHar}. Among the best of the latter are methods based on (co)reduction preprocessing \cite{mrozek:batko:09a} and discrete Morse theory \cite{harker:mischaikow:mrozek:nanda}.

With the parallel success of new applications and fast computations for homology, additional topological structures and techniques are poised to cross the threshold from theory to computation to application. Among the most promising is the theory of \style{sheaves}. Developed for applications in algebraic topology and matured under a string of breathtaking advances in algebraic geometry, sheaf theory is perhaps best described as a formalization of local-to-global transitions in Mathematics. The margins of this introductory section do not suffice to outline sheaf theory; rather, we present without detailed explanation three principal interpretations of a sheaf $\Sheaf$ over a topological space $\cX$ taking values in $\bR$-modules over some ring $\bR$:
\begin{enumerate}
\item A sheaf can be thought of as a {\em data structure} tethered to a space -- a assignment to open sets $\cV\subset \cU$ of $\cX$ a homomorphism $\Sheaf(\cU)\rightarrow\Sheaf(\cV)$ between $\bR$-modules -- the algebraic ``data'' over the subsets -- in a manner that respects composition and gluing (see \S\ref{sec:background}). Unlike in the case of a bundle, the data sitting atop subsets of $\cX$ can change dramatically from place-to-place.
\item A sheaf can be thought of as a {\em topological space} in and of itself, together with a projection map $\pi\colon\Sheaf\to \cX$ to the base space $\cX$. This {\em \'etale space} topologizes the data structure and motivates examining its topological features, such as (co)homology.
\item A sheaf can be thought of as a {\em coefficient system}, assigning to locations in $\cX$ the spatially-varying $\bR$-module coefficients to be used for computing cohomology. This representation of the space within the algebraic category of $\bR$-modules provides enough structure to compute cohomology with location-dependent coefficients.
\end{enumerate}
It is these multiple interpretations that portend the ubiquity of sheaves within applied topology. Though sheaves have long been recognized as useful data structures within certain branches of Computer Science (e.g., \cite{goguen1992sheaf}), sheaf cohomology has a number of emergent applications. These include:
\begin{enumerate}
\item {\bf Signal processing:} Sheaf cohomology recovers and extends the classical Nyquist-Shannon sampling theorem \cite{robinson2013nyquist}; {\em viz.}, reconstruction from a sample is possible if and only if the appropriate cohomology of an associated ambiguity sheaf vanishes.
\item {\bf Data aggregation:} Data aggregation over a domain can be performed via Euler integrals, an alternating reduction of the cohomology of an associated constructible sheaf over the domain \cite{curry:ghrist:robinson:12}.
\item {\bf Network coding:} Various problems in network coding (maximum throughput, merging of networks, rerouting information flow around a failed subnetwork) have interpretations as ranks of cohomologies of a sheaf over the network \cite{GhristHiraoka}.
\item {\bf Optimization:} The classical max-flow-min-cut theorem has a sheaf-theoretic analogue which phrases flow-values and cut-values as semimodule images of sheaf homology and cohomology respectively \cite{Krishnan,GhristKrishnan}.
\item {\bf Complexity:} A recent parallel to the Blum-Shub-Smale theory of complexity \cite{BSS89} has emerged for constructible sheaves \cite{Basu-2013}.
\end{enumerate}
These early examples of applications vary greatly in terms of the types of coefficients used (ranging from $\Z$ to $\R$-vector spaces to general commutative monoids) and the types of base spaces. In most applications, however, the relevant sheaves are of a particular discrete form. Topological spaces become computationally tractable substances through a discretization process: this most often takes the form of a simplicial or cell (or CW) complex. A similar modulation exists for sheaves -- a sheaf is called {\em constructible} with respect to a given stratification of the base space if the data assigned to each stratum is locally constant. We will work in the category of \style{cellular sheaves}, which are constructible with respect to a fixed regular CW stratification of the base space \cite{EAT}.

Motivated by these applications, we establish algorithms for the computation of sheaf cohomology. Our philosophy, inherited from other work on computational homology \cite{kaczynski:mischaikow:mrozek:04,mrozek:batko:09a,harker:mischaikow:mrozek:nanda} is that of reduction of the input structure to a smaller equivalent structure. We do so by means of \style{discrete Morse theory}, retooling the machinery to work for sheaves.

\subsection{Related and supporting work}

\subsubsection*{Sheaves}

A fair portion of the existing work on computational topology is naturally cast in the language of sheaves, providing novel paths for generalization. For example, Euler integration -- integration with respect to Euler characteristic as a valuation -- is sheaf-theoretic in nature and in origin, as per \cite{schapira1991operations,schapira1995tomography}. It is in fact the decategorification of the cohomology of sheaves associated to constructible functions (\cite{curry:ghrist:robinson:12} gives an exposition of this). More familiar to the reader will be persistent homology \cite{edelsbrunner:letscher:zomorodian:01, zomorodian:carlsson:05, carlsson:09}, which also has a sheaf-theoretic formulation as follows. The formal dual of a sheaf is a \style{cosheaf}; in the cellular category, these are quite useful \cite{EAT,DMP} and possess a homology theory \cite{curry}. The persistent homology of a filtration is the homology of a (certain) cosheaf over a cell complex homeomorphic to an interval \cite{curry}. Recent work on {\em well groups} associated to persistent homology has been expounded in terms of sheaves \cite{macpherson:patel}.


\subsubsection*{Discrete Morse theory:}
Discrete Morse theory \cite{forman98, chari} usually begins with the structure of a partial matching on the cells of a CW complex. The unmatched cells serve the same role as critical points do in smooth Morse theory while the matched cells furnish gradient-like trajectories between them. A Morse cochain  complex may be constructed from this data: its cochain groups are freely generated by the critical cells and the boundary operators may be derived from gradient paths. The fundamental result is that the Morse cochain complex so obtained is homologically equivalent to the original CW complex.

This basic idea has since been vastly generalized and adapted to purely algebraic situations  \cite{skoldberg, batzies:welker:00,kozlov05} with only the slightest vestige of its topological origins. One can impose a partial matching directly on the basis elements of a cochain complex and apply discrete Morse theory as usual. This approach has proved useful in the past when simplifying computation of homology groups of abstract cell complexes \cite{harker:mischaikow:mrozek:nanda} and the persistent homology groups of their filtrations \cite{mischaikow:nanda}.

\subsection{Problem statement and results}

Our problem centers on the computation of cellular sheaf cohomology. The initial inputs are a cellular sheaf $\Sheaf$ over a CW-complex $\cX$ taking values in free $\bR$-modules for some fixed ring $\bR$. This input is reprocessed into a cochain complex $F = (C^\bullet,d^\bullet)$ of free $\bR$-modules parameterized by a graded poset $(X,\leq)$ [see Definition \ref{defn:pos_cplx}]. Our main algorithm, {\tt Scythe} [see \S\ref{sec:algos}], constructs a $\Sheaf$-compatible acyclic matching $\Sigma$ on $(X,\leq)$ and suitably modifies the coboundary operators $d^\bullet$ in order to cut the original cochain complex down to its critical core while preserving its cohomology. The resulting smaller Morse cochain complex $F^\Sigma = (C_\Sigma^\bullet,d_\Sigma^\bullet)$ is parametrized by the poset of critical elements of $\Sigma$.

Let $\prec$ denote the covering relation in our graded poset $(X,\leq)$ and define for each $x \in X$ the set of immediate successors $x_+ = \setof{y \in X \mid x \prec y}$. The following parameters measure different aspects of the complexity of $F$:
\begin{enumerate}
\item let $n$ be the cardinality $|X|$ of the poset $X$,
\item let $p$ equal $\max_{x \in X}\setof{|x_+|}$,
\item assume that the maximum rank of $F(x)$ as an $\bR$-module is $d < \infty$ for $x \in X$,
\item assume that the matching $\Sigma$ produced by {\tt Scythe} has $m_k$ critical elements of dimension $k$ and define $\tilde{m} = \sum_km_k^2$, and
\item define $\omega$ to be the matrix multiplication exponent\footnote{That is, the complexity of composing two $d \times d$ matrices with $\bR$-entries is $\text{O}(d^\omega)$.} over $\bR$.
\end{enumerate}
Note that the first three numbers are input parameters, the fourth is an output parameter and the fifth is purely a property of the underlying coefficient ring $\bR$. Our main result is as follows.

\begin{thm*}
Let $F$ be a cochain complex of free $\bR$-modules over a graded poset $(X,\leq)$ and let $n, p, d, m$ and $\omega$ be the associated parameters defined above. Then, the time complexity of constructing the Morse complex $F^\Sigma$ via {\tt Scythe} is $\mathrm{O}(np\tilde{m}d^\omega)$ and the space complexity is $\mathrm{O}(n^2pd^2)$.
\end{thm*}

Section \ref{sec:background} contains background material on cellular sheaf theory and the fundamentals of discrete Morse theory. In \S\ref{sec:equiv} we provide explicit chain maps that induce isomorphisms on cohomology between the original and reduced complexes. Section \ref{sec:algos} contains a description of the algorithm {\tt Scythe}, a verification of its correctness and also a detailed complexity analysis which proves our main theorem above. Finally, in \S\ref{sec:apps} we develop distributed protocols for calculating traditional cohomology groups of a given space by recasting the computations in appropriate sheaf-theoretic frameworks.


\section{Background}\label{sec:background}

In this section we survey preliminary material pertaining to cellular sheaves \cite{shepard,curry,vybornov:2012} and a purely algebraic version of discrete Morse theory \cite{skoldberg,batzies:welker:00,kozlov05}. Throughout this paper, $\bR$ denotes a fixed coefficient ring with identity $1_\bR$ while $\N$ and $\Z$ denote the natural numbers and integers respectively.

\subsection{Cellular Sheaves and their Cohomology}

Let $\cX$ be a finite regular CW complex consisting of cells and their attaching maps \cite{munkres, spanier}. For each $n \in \N$ the subcollection of $n$-dimensional cells will be  written $\cX^n$. Given cells $\sigma$ and $\tau$ of $\cX$, we write $\sigma \leq \tau$ to indicate the face relation in $\cX$.
Finally, for each pair of cells $\sigma$ and $\tau$ in $\cX$, the quantity $[\sigma:\tau] \in \Z$ is defined to equal
\begin{itemize}
\item $+1$ if $\sigma \leq \tau$, $\dim \sigma = \dim \tau - 1$, and the local orientations of their attaching maps agree;
\item $-1$ if $\sigma \leq \tau$, $\dim \sigma = \dim \tau - 1$, and the local orientations disagree; and
\item $0$ otherwise.
\end{itemize}
It follows from the usual boundary operator axiom that the following relation must hold across each pair of cells $\sigma, \tau \in \cX$:
\begin{align}
\label{eqn:orientation}
\sum_{\sigma \leq \lambda \leq \tau} \hsp [\sigma:\lambda][\lambda:\tau] = 0.
\end{align}

\begin{defn}
\label{defn:cell_sheaf}
{\em
A {\em cellular sheaf} $\Sheaf$ over $\cX$ assigns to each cell $\sigma$ of $\cX$ an $\bR$-module $\Sheaf(\sigma)$ and to each face relation $\sigma \leq \tau$ an $\bR$-linear {\em restriction map} $\Sheaf_{\sigma\tau}:\Sheaf(\sigma) \to \Sheaf(\tau)$ subject to the following compatibility condition: whenever $\sigma \leq \lambda \leq \tau$ in $\cX$, we have $\Sheaf_{\lambda\tau} \circ \Sheaf_{\sigma\lambda} = \Sheaf_{\sigma\tau}$.
}
\end{defn}

Simple examples of sheaves include the following:
\begin{enumerate}
\item The {\em constant sheaf}, $\bR_\cX$, assigns the coefficient ring $\bR$ to each cell of $\cX$ and the identity restriction map $1_\bR:\bR \to \bR$ to each face relation.
\item The {\em skyscraper sheaf} over a single cell $\sigma$ of $\cX$ is a sheaf, $\bR_\sigma$, that evaluates to $\bR$ on $\sigma$ and is zero elsewhere, with all restriction maps being zero.
\item An analogue of the skyscraper sheaf over a subcomplex $\cA\subset\cX$ evaluates to $\bR$ on all cells of $\cA$ and zero elsewhere. The restriction maps are zero except for the identity map from a cell in $\cA$ to a face. This sheaf is best described as the {\em pushforward} $\iota_*\bR_\cA$ of the constant sheaf on $\cA$ induced by the inclusion map $\iota\colon\cA\hookrightarrow\cX$. This is {\em not} the same as the sum of skyscraper sheaves over the cells of $\cA$, since the restriction maps are not all zero.
\end{enumerate}

Given any cellular sheaf $\Sheaf$ on $\cX$, we define the $n$-th {\em cochain group over $\Sheaf$} to be the  direct sum of the $\bR$-modules assigned by $\Sheaf$ to the $n$-dimensional cells. That is,
\[
C^n(\cX;\Sheaf) = \bigoplus_{\sigma \in \cX^n} \Sheaf(\sigma).
\]
The $n$-th {\em coboundary operator} $\delta^n:C^n(\cX;\Sheaf) \to C^{n+1}(\cX;\Sheaf)$ is completely determined by the following block action. Given $\sigma \in \cX^n$ and $\tau \in \cX^{n+1}$, the component of $\delta^n$ from $\Sheaf(\sigma)$ to $\Sheaf(\tau)$ precisely equals $[\sigma:\tau]\Sheaf_{\sigma\tau}$ and so we obtain a sequence of $\bR$-modules
\[
0 \to C^0(\cX;\Sheaf) \mapright{\delta^0} C^1(\cX;\Sheaf) \mapright{\delta^1} C^2(\cX;\Sheaf) \mapright{\delta^2} \cdots
\]
It follows from a routine calculation involving (\ref{eqn:orientation}) and the compatibility condition of Definition \ref{defn:cell_sheaf} that $\delta^{n}\circ\delta^{n-1} = 0$ for all $n \in \N$ and hence that $(C^\bullet(\cX;\Sheaf), \delta^\bullet)$ is a cochain complex.

\begin{defn}
\label{defn:sheaf_cohom}
{\em
Let $\Sheaf$ be a cellular sheaf on $\cX$. The {\em cohomology of $\cX$ with $\Sheaf$ coefficients} is defined to be the cohomology of the cochain complex $(C^\bullet(\cX;\Sheaf), \delta^\bullet)$. More precisely,
\[
H^n(\cX;\Sheaf)  = \frac{\ker \delta^n}{\img \delta^{n-1}}.
\]
}
\end{defn}
The reader may interpret $H^\bullet(\cX;\cF)$ as the cohomology of the data $\cF$ over $\cX$. The simple examples of sheaves listed above have the following cohomologies:
\begin{enumerate}
\item The constant sheaf $\bR_\cX$ on $\cX$ has cohomology $H^\bullet(\cX;\bR_\cX)\cong H^\bullet(\cX;\bR)$ equal to ordinary cohomology in $\bR$ coefficients.
\item The skyscraper sheaf $\bR_\sigma$ on $\cX$ has cohomology $H^k(\cX;\bR_\sigma)\cong \bR$ when $k = \dim \sigma$ and zero otherwise, illustrating that a sheaf can have trivial cohomology even if the underlying base space is noncontractible.
\item The pushforward sheaf $\iota_*\bR_\cA$ has cohomology $H^\bullet(\cX;\iota_*\bR_\cA)\cong H^\bullet(\cA;\bR)$, illustrating that a sheaf can have complicated cohomology even if the underlying base space is contractible.
\end{enumerate}

Of course, more intricate examples abound and are the impetus for an effective algorithm for computation.

\subsection{Morse Theory for Parametrized Cochain Complexes}

Forman's work on Morse theory for CW complexes \cite{forman98} has been extended to a purely algebraic framework by Batzies and Welker \cite{batzies:welker:00}, Kozlov \cite{kozlov05}, and (in greatest generality) by Sk\"oldberg \cite{skoldberg}. The central idea is to exploit invertible restriction maps in order to produce a smaller cochain complex with isomorphic cohomology. In order to establish notation compatible with an algorithmic treatment, we provide a brief overview of the main results here.

Recall that given two elements $x,y$ in a poset $(X,\leq)$ we say that $y$ {\em covers} $x$ whenever $x < y$ and $\setof{z \in X \mid x < z < y} = \varnothing$. We denote this covering relation by $x \prec y$ and call a poset $(X,\leq)$ {\em graded} if it admits a partition $X = \bigcup_{n \in N}X_n$ into subsets indexed by a dimension so that if $x \prec y$ then $\dim y = \dim x + 1$. All graded posets in sight are assumed to be finite\footnote{When striving for greater generality, one replaces this requirement by the following local finiteness hypothesis on the covering relation: each $x \in X$ can have only finitely many $y$ so that $y \prec x$ or $x \prec y$.}.

\begin{defn}
\label{defn:pos_cplx}
{\em
A {\em parametrization} $F$ of a cochain complex $(C^\bullet,d^\bullet)$ of $\bR$-modules over a graded poset $(X,\leq)$ assigns to each $x \in X$  an $\bR$-module $F(x)$ and to each covering relation $x \prec y$ a linear map $F_{xy}:F(x) \to F(y)$ so that for all dimensions $n \in \N$,
\begin{enumerate}
\item $C^n = \bigoplus_{x \in X_n} F(x)$, and
\item the block of $d^n:C^n \to C^{n+1}$ from $F(x)$ to $F(y)$ is precisely $F_{xy}$.
\end{enumerate}
By convention, we require $F_{xy} = 0$ whenever $x \not\prec y$.
}
\end{defn}

Cochain complexes parametrized over posets are the basic objects on which discrete Morse theory operates. Before introducing the details, we remark that the cells of a finite regular CW complex $\cX$ comprise a graded poset over which the cochain complex $(C^\bullet(\cX;\Sheaf),\delta^\bullet)$ associated to any sheaf $\Sheaf$ is naturally parametrized. Throughout the remainder of this section, we fix a parametrization $F$ of a cochain complex $(C^\bullet,d^\bullet)$ over a graded poset $(X,\leq)$.

The following definition goes back to the work of Chari \cite{chari}
\begin{defn}
\label{defn:matching}
{\em
A {\em partial matching} on $(X,\leq)$ is a subset $\Sigma \subset X \times X$ of pairs subject to the following axioms:
\begin{enumerate}
\item {\bf dimension:} if $(x,y) \in \Sigma$ then $x \prec y$, and
\item {\bf partition:} if $(x,y) \in \Sigma$ then neither $x$ nor $y$ belong to any other pair in $\Sigma$.
\end{enumerate} Moreover, $\Sigma$ is called {\bf acyclic} if the transitive closure of the relation $\lhd$ defined on pairs in $\Sigma$ by
\[
(x,y) \lhd (x',y') \text{ if and only if } x \prec y',
\]
generates a partial order.
}
\end{defn}

We call an acyclic matching $\Sigma$ on $(X,\leq)$ {\bf compatible} with the parametrization $F$ if for each pair $(x,y) \in \Sigma$ the associated linear map $F_{xy}:F(x) \to F(y)$ is invertible. Let $\Sigma$ be such a compatible acyclic matching on $(X,\leq)$ and denote by $M$ the {\em critical} unpaired elements:
\[
M = \setof{m \in X \mid (m,z) \text{ and } (z,m) \text{ are not in }\Sigma\text{ for any }z \in X}.
\]
A {\em gradient path} $\gamma$ of $\Sigma$ is a strictly $\lhd$-increasing sequence $(x_j,y_j)_1^J \subset \Sigma$ arranged as follows:
\[
\gamma = y_1 \succ x_1 \prec y_2 \succ x_2 \prec \cdots \prec y_J \succ x_J,
\]
and its {\em coindex} $F_\gamma:F(y_1) \to F(x_J)$ is the linear map given by
\begin{align}
\label{eqn:coindex}
F_\gamma =  \left(-F_{x_Jy_J}^{-1}\right) \circ F_{x_{J-1}y_J} \circ \cdots  \circ F_{x_1y_2} \circ \left(-F_{x_1y_1}^{-1}\right).
\end{align}
For each gradient path $\gamma = (x_j,y_j)_1^J$, we write $\src_\gamma = y_1$ and $\tgt_\gamma = x_J$ to indicate the source (first) and target (last) elements. Given critical elements $m, m' \in M$, the path $\gamma$ is said to flow from $m$ to $m'$ whenever the covering relations $m \prec \src(\gamma)$ and $\tgt(\gamma) \prec m'$ both hold; and a new linear map $F^\Sigma_{mm'}:F(m) \to F(m')$ may be defined by:
\begin{align}
\label{eqn:morserest}
F^\Sigma_{mm'} = F_{mm'} + \sum_\gamma F_{\tgt_\gamma m'} \circ F_\gamma \circ F_{m\src_\gamma},
\end{align}
where the sum is taken over all gradient paths $\gamma$ of $\Sigma$ flowing from $m$ to $m'$. If we write $m <_\Sigma m'$ whenever at least one such path exists, then it follows easily from the acyclicity of $\Sigma$ that the transitive closure of $<_\Sigma$ furnishes a partial order on $M$ which is graded by dimension.

\begin{defn}
\label{defn:morsedata}
{\em
The {\em Morse data} associated to $\Sigma$ consists of the poset $(M,\leq_\Sigma)$ of critical elements along with a sequence of $\bR$-modules
\[
0 \to C_\Sigma^0 \mapright{d_\Sigma^0} C_\Sigma^1 \mapright{d_\Sigma^1} C_\Sigma^2 \mapright{d_\Sigma^2} \cdots
\]
where $C^n_\Sigma = \bigoplus_{m \in M_n} F(m)$ and the block of $d_\Sigma^n:C_\Sigma^n \to C_\Sigma^{n+1}$ from $F(m)$ to $F(m')$ is $F^\Sigma_{mm'}$.
}
\end{defn}

The following theorem is (dual to) the main result of algebraic Morse theory.
\begin{thm}[Sk\"oldberg, \cite{skoldberg}]
\label{thm:dmtcohom}
Let $F$ parametrize a cochain complex $(C^\bullet,d^\bullet)$ of $\bR$-modules over a graded poset $(X,\leq)$ and let $\Sigma$ be a compatible acyclic matching. Then, the Morse data $(C_\Sigma^\bullet,d_\Sigma^\bullet)$ is a cochain complex parametrized over $(M,\leq_\Sigma)$ by $F^\Sigma$. Moreover, there are $\bR$-module isomorphisms
\[
H^n(C,d) \cong H^n(C_\Sigma,d_\Sigma),
\]
on cohomology for each dimension $n \in \N$.
\end{thm}
In the next section we provide a new proof of Theorem \ref{thm:dmtcohom} by constructing explicit cochain equivalences. This proof leads to a recipe for simplifying cohomology computation for an arbitrary cellular sheaf $\Sheaf$ given the existence of efficient techniques for constructing compatible matchings and the Morse data. One imposes an acyclic matching $\Sigma$ on the graded poset of cells in the underlying regular CW complex $\cX$ so that for each $(\sigma,\tau) \in \Sigma$ the restriction map $\Sheaf_{\sigma\tau}$ is invertible. If the set $M$ of critical cells is much smaller than $X$, then one simply computes the cohomology of the smaller cochain complex $(C_\Sigma^\bullet,d_\Sigma^\bullet)$.

\section{The Cohomological Morse Equivalence}\label{sec:equiv}

Let $F$ be a parametrization for a cochain complex $(C^\bullet,d^\bullet)$ over a graded poset $(X,\leq)$ and assume that $\Sigma$ is a compatible acyclic matching on $X$. We prove Theorem \ref{thm:dmtcohom} via an inductive argument by removing one $\Sigma$-pair at a time from $X$. By suitably updating the parametrization near the removed pair at each step, it is possible to preserve the cohomology until one converges to the Morse parametrization $F^\Sigma$ over the poset $(M,\leq_\Sigma)$ of critical elements.

\subsection{The Reduction Step} \label{ssec:red}

The central idea of reducing a cell pair from a CW complex while preserving its homotopy type (and hence, its cohomology) goes back to the work of Whitehead on combinatorial homotopy \cite{whitehead}. Here we present a suitable version of this reduction step adapted for cellular sheaves and efficient algorithms.

Fix $(x^\red,y^\red) \in \Sigma$ and define $X^\red = X \setminus \setof{x^\red,y^\red}$. A graded partial order $\leq_\red$ may be defined on $X^\red$ via the following covering relation: given any cells $w$ and $z$ in $X^\red$, we have $w \prec_\red z$ if either $w \prec z$ in $X$ or if $w \prec y^\red \succ x^\red \prec z$ in $X$. One obtains a new parametrization $F^\red$ over the reduced poset $(X^\red,\leq_\red)$ as follows: $F^\red(w) = F(w)$ for all $w \in X^\red$, and for each covering relation $w \prec_\red z$ we have the linear map $F^\red_{wz}:F(w) \to F(z)$ given by
\begin{align}
\label{eqn:red_param}
F^\red_{wz} = F_{wz} - F_{x^\red z} \circ F_{x^\red y^\red}^{-1} \circ F_{w y^\red}.
\end{align}
A routine calculation shows that $F^\red$ parametrizes a cochain complex which we denote by $(C_\red^\bullet,d_\red^\bullet)$; moreover, $\Sigma$ restricts to an acyclic matching $\Sigma^\red$ on $(X^\red,\leq_\red)$.

\begin{prop}
\label{prop:samemors}
Given the restricted acyclic matching $\Sigma^\red$ defined above,
\begin{enumerate}
\item $\Sigma^\red$ is compatible with the reduced parametrization $F^\red$, and
\item the Morse data associated to $\Sigma^\red$ is identical to that of $\Sigma$.
\end{enumerate}
\end{prop}
\begin{proof}
In fact, for any $(x,y) \in \Sigma^\red$ there is an equality $F_{xy} = F^\red_{xy}$ by (\ref{eqn:red_param}) -- otherwise, we violate the acyclicity of $\Sigma$ as follows. By (\ref{eqn:red_param}) the non-zeroness of $F^\red_{xy} - F_{xy}$ implies that $F_{xy^\red}$ and $F_{x^\red y}$ do not vanish, which leads to the contradiction $(x,y) \lhd (x^\red,y^\red) \lhd (x,y)$. To prove the second assertion, first note that the critical elements of $\Sigma$ and $\Sigma^\red$ are identical; and since $F(m) = F^\red(m)$ for each critical $m$,  one obtains an equality of cochain groups $C_\Sigma^n = C_{\Sigma^\red}^n$ for each dimension $n \in N$. Thus, we turn our attention to the linear maps $d_\Sigma$ and $d_{\Sigma^\red}$. Given any gradient path $\gamma^\red$ of $\Sigma^\red$, say
\[
\gamma^\red = y_1 \succ_\red x_1 \prec_\red \cdots \prec_\red y_J \succ_\red x_J,
\]
it follows by acyclicity of $\Sigma$ that there is at most one index $j \in \setof{1,\ldots,J-1}$ for which we may have $(x_j,y_j) \lhd (x^\red,y^\red) \lhd (x_{j+1},y_{j+1})$. Returning to our path $\gamma^\red$, we therefore conclude that there are only two possibilities. Either there is no index $j$ at which the removed pair $(x,y)$ might fit, in which case $\gamma^\red$ is also a path of $\Sigma$ with $F_{\gamma^\red} = F^\red_{\gamma^\red}$ from (\ref{eqn:coindex}). Alternately, there is a single such index $j$, in which case $\Sigma$ may have as its paths both $\gamma^\red$ and the unique augmented path $\gamma$ given by introducing the removed pair $(x^\red,y^\red)$ in the appropriate spot:
\[
\gamma = y_1 \succ x_1 \prec \cdots \prec y_j \succ x_j \prec y^\red \succ x^\red \prec y_{j+1} \succ \cdots \prec y_J \succ x_J.
\]
It follows from a quick calculation that $F^\red_{\gamma^\red} = F_{\gamma^\red} + F_\gamma$. In both cases, the sum of coindices over all paths (and hence each block of $d^\bullet_\Sigma$) is preserved. Using this information in (\ref{eqn:morserest}) and Definition \ref{defn:morsedata} concludes the argument.
\end{proof}

As a consequence of this proposition, the Morse complex $(C_\Sigma^\bullet,d_\Sigma^\bullet)$ remains invariant under the reduction step. It remains to show that cohomology is preserved when passing from $F$ to the reduced parametrization $F^\red$.

\subsection{Cochain Equivalences} \label{ssec:coch}

For each $n \in N$, define the linear map $\psi^n:C^n \to C_\red^n$ by the following block action. For $w \in X_n$ and $z \in X^\red_n$, the block $\psi_{wz}:F(w) \to F(z)$ is given by:
\begin{align}
\label{eqn:psi}
\psi_{wz} = \begin{cases}
																  - F_{x^\red z} \circ F_{x^\red y^\red}^{-1}	& w = y^\red, \\
																	\text{id}_{F(w)}																							& w = z, \\
																	0 																																		& \text{otherwise.}
														 \end{cases}
\end{align}

\begin{lem}
\label{lem:psicoch}
$\psi^\bullet : C^\bullet \to C_\red^\bullet$ is a cochain map. That is, $\psi^{n+1} \circ d^n = d_\red^n \circ \psi^n$ for each $n \in \N$.
\end{lem}
\begin{proof}
Given $w \in X_n$ and $z \in X^\red_{n+1}$ we show that the blocks of $\psi^{n+1}\circ d^n$ and $d_\red^n\circ \psi^n$ from $F(w)$ to $F^\red(z) = F(z)$ are identical. More precisely, we wish to establish the following:
\[
\sum_{w' \in X_{n+1}} \hsp \psi_{w'z} \circ F_{ww'} =  \hsp  \sum_{z' \in X^\red_n} F^\red_{z' z} \circ \psi_{w z'}.
\]
By (\ref{eqn:psi}) we note that the left side is nonzero only for $w' = z$ or for $w' = y^\red$. Combining these contributions, the left side evaluates to $F_{wz} + \psi_{y^\red z}\circ F_{wy^\red}$ which equals $F^\red_{wz}$. Similarly, the right side of the identity above also reduces to $F^\red_{wz}$ immediately at least when $w \neq y^\red$, so it now suffices to show that this right side equals $F^\red_{y^\red z}$ whenever $w = y^\red$. In this case, we calculate
\[
\sum_{z' \in X^\red_n} F^\red_{z' z} \circ \psi_{y^\red z'} = -\hsp \sum_{z' \in X^\red_n} F^\red_{z'z} \circ F_{x^\red z'} \circ F_{x^\red y^\red}^{-1}
\]
Expanding $F^\red_{z'z}$ via (\ref{eqn:red_param}) and distributing terms gives
\[
- \hsp \sum_{z' \in X^\red_n}  F_{z'z} \circ F_{x^\red z'}\circ F_{x^\red y^\red}^{-1}
+ \hsp \sum_{z' \in X^\red_n}  F_{x^\red z} \circ F_{x^\red y^\red}^{-1} \circ F_{z'y^\red} \circ F_{x^\red z'} \circ  F_{x^\red y^\red}^{-1}.
\]
The second sum is zero: since $x^\red \prec y^\red$ by Definition \ref{defn:matching}, there is no $z' \in X^\red$ satisfying $x^\red \prec z' \prec y^\red$ and so the summand is always trivial. Finally, one can use the fact that $d^\bullet$ is a coboundary operator -- in particular, that $\sum_{z' \in X_n} F_{z'z} \circ F_{x^\red z'} = 0$ -- to show that the first sum equals $F^\red_{y^\red z}$ as desired.

\end{proof}

We now require a cochain map in the other direction. To this end, define $\phi^n:C_\red^n \to C^n$ by the following block action $\phi_{zw}: F(z) \to F(w)$ for each $z \in X^\red_n$ and $w \in X_n$:
\begin{align}
\label{eqn:phi}
\phi_{zw}      = \begin{cases}
														-  F_{x^\red y^\red}^{-1} \circ F_{z y^\red} & w = x^\red, \\
														\text{id}_{F(w)} & w = z, \\
														0 & \text{otherwise}.
												\end{cases}
\end{align}

\begin{lem}
\label{lem:phicoch}
$\phi^\bullet : C_\red^\bullet \to C^\bullet$ is a cochain map. That is, $\phi^{n+1} \circ d_\red^n = d^n \circ \phi^n$ for each $n \in \N$.
\end{lem}
\begin{proof}
The argument proceeds very similarly to the one in the proof of Lemma \ref{lem:psicoch}. Given $z \in X^\red_n$ and $w \in X_{n+1}$, we establish a block-equivalence by showing that the following identity holds:
\[
\sum_{z' \in X^\red_{n+1}} \hsp \phi_{z'w} \circ F^\red_{zz'} =  \hsp  \sum_{w' \in X_n} F_{w'w} \circ \phi_{zw'}.
\]
By (\ref{eqn:phi}) we note that the right side is nontrivial only when $w' = x^\red$ or when $w'=z$, and hence it reduces to $F_{zw} + F_{x^\red z}\circ \phi_{zx^\red}$, which equals $F^\red_{zw}$. The left side also evaluates to the same quantity whenever it is nontrivial provided that $w \neq x^\red$. On the other hand, if $w = x^\red$ then the left side becomes
\[
\sum_{z' \in X^\red_{n+1}} \hsp \phi_{z'x^\red} \circ F^\red_{zz'} = - \hsp \sum_{z' \in X^\red_{n+1}} \hsp F_{x^\red y^\red}^{-1} \circ F_{z' y^\red} \circ F^\red_{zz'}.
\]
Expanding $F^\red_{zz'}$ via (\ref{eqn:red_param}) and distributing terms yields
\[
- \hsp \sum_{z' \in X^\red_{n+1}} \hsp F_{x^\red y^\red}^{-1} \circ F_{z' y^\red} \circ F_{zz'}
+ \hsp \sum_{z' \in X^\red_{n+1}} \hsp F_{x^\red y^\red}^{-1} \circ F_{z' y^\red} \circ F_{x^\red z'} \circ F_{x^\red y^\red}^{-1} \circ F_{zy^\red}.
\]
The second sum above is always zero, since $(x^\red,y^\red) \in \Sigma$ implies $x^\red \prec y^\red$ and hence there is no $z' \in X^\red$ with $x^\red \prec z' \prec y^\red$. Finally, the first sum reduces to $F^\red_{x^\red z}$ since $d^\bullet$ is a coboundary operator, and hence $\sum_{z' \in X_{n+1}} F_{z' y^\red} \circ F_{zz'} = 0$.
\end{proof}

It is easy to verify that $\psi^n \circ \phi^n$ is the identity map on $C_\red^n$ for each $n \in \N$, so in order to conclude that $\psi^\bullet$ and $\phi^\bullet$ are cochain equivalences it suffices to construct a cochain homotopy $\Theta^n:C^n \to C_\red^{n-1}$ between $\phi^n \circ \psi^n$ and the identity on $C^n$. The following result completes our proof of Theorem \ref{thm:dmtcohom}.

\begin{lem}
The linear maps $\Theta^n:C^n \to C^{n-1}$ defined by the block action
\[
\Theta_{ww'}    = \begin{cases}
																F_{x^\red y^\red}^{-1} & w' = x^\red \text{ \em and } w = y^\red,\\
																0 & \text{\em otherwise,}
														\end{cases}
\]
constitute a cochain homotopy between $\phi^n\circ\psi^n$ and the identity on $C^n$ for each dimension $n \in \N$.
\end{lem}
\begin{proof}
By definition, it suffices to show $\Theta^{n+1} \circ d^n + d^{n-1} \circ \Theta^n = \text{id}_{C^n} - \phi^n \circ \psi^n$. By (\ref{eqn:psi}) and (\ref{eqn:phi}) we note that $\phi \circ \psi$ has the following block action $F(w) \to F(\ow)$ for $w, \ow \in X_n$:
\[
(\phi \circ \psi)_{w\ow} = \begin{cases}
         																			 - F_{x^\red y^\red}^{-1} \circ F_{w y^\red} 		& \ow = x^\red, \\
					  																   - F_{x^\red \ow} \circ F_{x^\red y^\red}^{-1}  & w = y^\red, \\						
																							                                         \text{id}_{F(w)} 					&  w = \ow \in X^\red, \\
																																												0   															 &  \text{otherwise.}
																				\end{cases}
\]
A simple calculation confirms that $\id_{C^n} - \Theta^{n+1} \circ d^n - d^{n-1} \circ \Theta^n$ has precisely the same block action and concludes the proof.
\end{proof}

\section{Algorithms}\label{sec:algos}

In this section we describe our algorithm {\tt Scythe} which constructs an acyclic matching $\Sigma$ on $(X,\leq)$ and iteratively implements the reduction step of \S \ref{sec:equiv} in order to reduce a poset-parametrized cochain complex down to the Morse parametrization. Before turning to the details, we recall our main result. Let $F$ be a parametrization of a cochain complex $(C^\bullet,d^\bullet)$ of free $\bR$-modules over a graded poset $(X,\leq)$ whose covering relation is denoted by $\prec$ as usual. For each $x \in X$ we define $x_+ = \setof{y \in X \mid x \prec y}$ and similarly $x_- = \setof{y \in X \mid x \succ y}$. Assume that the acyclic matching imposed by {\tt Scythe} is called $\Sigma$. Recall from \S\ref{sec:intro} the parameters $n=|X|$, $p=\max_{x \in X}\setof{|x_+|}$, $d=\max_x\rank(F(x))$, $\tilde{m} = \sum_km_k^2$, and $\omega$.

Note that $n$, $p$, and $d$ are input parameters. The net critical elements cardinality $\tilde{m}$ is an output parameter, and the multiplication exponent $\omega$ is purely a property of the underlying coefficient ring $\bR$. The remainder of this section is dedicated to proving our main theorem.

\begin{thm}
\label{thm:main}
Let $F$ parametrize a cochain complex of $\bR$-modules over a graded poset $(X,\leq)$ and let $n, p, d, \tilde{m}$ and $\omega$ be the parameters defined above. Then, the time complexity of constructing the Morse parametrization $F^\Sigma$ via {\tt Scythe} is $\mathrm{O}(np\tilde{m}d^\omega)$ and the space complexity is $\mathrm{O}(n^2pd^2)$.
\end{thm}

\subsection{Description and Verification}

\begin{table}
\centering
{\bf Algorithm: }{\tt Scythe} \label{alg:morsred} \\
{\bf In: }A parametrization $F$ of a cochain complex over a graded poset $(X,\leq)$ \\
{\bf Out: } Transforms $F$ to the Morse parametrization $F^\Sigma$, \\
where $\Sigma$ is an $F$-compatible acyclic matching on $(X,\leq)$.\\
{
\begin{tabular}{|c|l|}
\hline
01 & {\bf define} a queue \Que of $X$-elements \\
02 & {\bf while} $X$ has non-critical elements \\
03 & \tab {\bf select} a minimal non-critical element $c$ of $X$\\
04 & \tab {\bf mark} $c$ as critical \\
05 & \tab {\bf set} \Que $= \varnothing$ \\
06 & \tab {\bf enqueue} $c$ into \Que \\
07 & \tab {\bf while} \Que  is nonempty \\
08 & \tab \tab {\bf dequeue} $y$ from \Que \\
09 & \tab  \tab {\bf if} $y_-$ has exactly one non-critical $x$ with $F_{xy}$ invertible \\
10 & \tab \tab \tab {\bf enqueue} $x_+ \setminus \setof{y}$ into \Que \\
11 & \tab \tab \tab {\tt ReducePair}$(x,y)$ \\
12 & \tab  \tab {\bf end if} \\
13 & \tab \tab {\bf enqueue} $y_+$ into \Que \\
14 & \tab  {\bf end while} \\
15 &  {\bf end while} \\
\hline
\end{tabular}
}
\end{table}

The central idea behind our algorithm is derived from iterated breadth-first search\footnote{In principle, any method for constructing acyclic partial matchings on graded posets will suffice, provided that it ensures sheaf-compatibility by only matching cell pairs whose restriction maps are invertible.} and has been exploited on several occasions in similar but less general computational contexts \cite{mrozek:batko:09a,mischaikow:nanda}. A minimal element $c \in X$ is chosen arbitrarily and declared critical, and elements $y \in c_+$ are scoured for possible pairings. Such a $y$ comprises a viable candidate for pairing if there is a unique uncritical element $x \in y_-$ so that $F_{xy}$ is invertible. As each such pair is found, the reduction step of \S\ref{ssec:red} is applied and both the poset $X$ as well as the parametrization $F$ are locally modified near the reduced pair $(x,y)$ by the subroutine {\tt ReducePair}. The removal of these pairs creates the possibility of new viable candidates for pairings, and we keep track of them using a queue data structure.

Given a pair $x^\red \prec y^\red$ of elements in $X$ with $F_{x^\red y^\red}$ invertible, {\tt ReducePair} performs the reduction step from \S\ref{ssec:red}. The key step of this subroutine is Line 04 which corresponds to updating $F$-values as described in (\ref{eqn:red_param}). Minor modifications to {\tt ReducePair} along with a few additional data structures would also allow us to catalog and store the cochain equivalences $\psi$ and $\phi$ as described in \S\ref{ssec:coch}.

\begin{prop}
\label{prop:verify}
The collection of those $(x,y) \in X \times X$ for which {\tt ReducePair}$(x,y)$ is invoked in Line {\em 12} of {\tt Scythe} constitutes an $F$-compatible acyclic matching $\Sigma$ on $(X,\leq)$.
\end{prop}
\begin{proof}
The compatibility of the pairing with the parametrization $F$ is enforced in Line 10 of {\tt Scythe} where we check for the invertibility of $F_{xy}$. The partial matching axioms of Definition \ref{defn:matching} are easily seen to be satisfied, so we focus here on proving that $\Sigma$ is acyclic. Returning again to Line 10, note that we only make a pairing $(x,y)$ when $x$ is the last remaining uncritical element in $y_-$. Now, any pair $(x',y')$ for which $(x',y') \lhd (x,y)$ must by definition satisfy $x' \prec y$, or equivalently, $x' \in y_-$. Since any such $x'$ is manifestly uncritical, it must already have been removed from $X$ along with its paired element $y'$ before the current pair $(x,y)$ was removed. Thus, the order of pair removal is monotonic with respect to $\lhd$ and so the collection of removed pairs generates an acyclic matching on $X$.
\end{proof}

It follows immediately from the preceding proposition and the machinery developed in \S\ref{sec:equiv} that the input parametrization $F$ is modified in-place to the Morse parametrization $F^\Sigma$: the input poset $(X,\leq)$ is reduced to the critical poset $(M,\leq_\Sigma)$ and the coboundary operator is suitably updated one pair at a time.

\subsection{Complexity Analysis}

\begin{table}
\centering
{\bf Algorithm: }{\tt ReducePair} \label{alg:psi} \\
{\bf In: }A pair $(x^\red,y^\red) \in X \times X$ with $x^\red \prec y^\red$ and $F_{x^\red y^\red}$ invertible \\
{\bf Out: } Modifies $F$ according to the reduction step  \\
{
\begin{tabular}{|c|l|}
\hline
01 & {\bf for each} $z \in x^\red_+ \setminus \setof{y^\red}$ \\
02 & \tab {\bf for each} $w \in y^\red_- \setminus \setof{x^\red}$ \\
03 & \tab \tab {\bf set} $w \prec z$ \\
04 & \tab \tab {\bf replace} $F_{wz}$ by $F_{wz} - F_{x^\red z} \circ F_{x^\red y^\red}^{-1} \circ F_{wy^\red}$ \\
05 & \tab {\bf end for} \\
06 & {\bf end for} \\
07 & {\bf remove} $x^\red$ and $y^\red$ from $X$ \\
\hline
\end{tabular}
}
\end{table}

Before performing a thorough analysis of {\tt Scythe} in terms of the complexity parameters introduced in the previous section, we briefly describe some simplifying assumptions. First, the Queue data structure must be managed so that the inner {\bf while} loop spanning Lines 02 through 14 actually terminates. Whenever an element of $X$ is added to the Queue, it is flagged so that it may not be enqueued again in that iteration of the inner {\bf while} loop. But each time the Queue is reinitialized in Line 05, all these flags are cleared. Moreover, we ensure that aside from the critical cell $c$ chosen in Line 03 of {\tt Scythe}, no other critical cells are enqueued. Finally, for the purposes of analyzing complexity we make the simplifying assumption that we only enqueue those elements of $X$ whose dimension exceeds $\dim c$ by $1$. Although this restriction is unnecessary (and indeed, detrimental to performance) in practice, it greatly simplifies the complexity analysis.

Note that the time complexity of calling {\tt ReducePair} with input $(x,y)$ where $\dim x = k$ is $\text{O}(pm_kd^\omega)$ as follows. The cardinality of $x_+\setminus\setof{y}$ is at most $p$ by assumption, and since the set $y_- \setminus\setof{x}$ only has critical elements by Line 09 of {\tt Scythe}, its cardinality does not exceed $m_k$. For each pair $w$ and $z$ of elements from these sets, the matrix algebra of Line 04 incurs a further cost of $\text{O}(d^\omega)$: the cost of matrix addition is dominated by the costs of inversion and multiplication, which are $\text{O}(d^\omega)$ by assumption. Since the inner {\bf while} loop runs at most $n$ times, its total time complexity is $\text{O}(npm_kd^\omega)$ where $k = \dim c$ by virtue of our restricted queuing strategy. Finally, since the outer {\bf while} loop executes precisely once per $k$-dimensional critical element, the total complexity of {\tt Scythe} evaluates to $\text{O}(np\tilde{m}d^\omega)$ as claimed in Theorem \ref{thm:main}.

\begin{rem}
{\em
The following observations have practical significance when simplifying computation of cellular sheaf cohomology via discrete Morse theory.
\begin{enumerate}
\item Since the output Morse parametrization $F^\Sigma$ generated by {\tt Scythe} is again a cochain complex of $\bR$-modules parametrized by a poset, it is possible to {\em iterate} the simplification scheme. In particular, one may impose an $F^\Sigma$-compatible acyclic matching on the critical poset $(M,\leq_\Sigma)$ and so forth, until the Morse parametrization stabilizes. This stabilization is caused by the eventual depletion of cell pairs which may be compatibly matched. In particular, if there are no invertible sub-blocks in the matrix representation of the Morse coboundary operator, then no further cell pairs may be matched. 
\item There is an obvious {\em dual} algorithm, {\tt CoScythe}, which processes $(X,\leq)$ from the top-down. In particular, a maximal element $c \in X$ may be initially chosen as critical and one may then search for pairings in the set $c_-$ of elements covered by $c$.
\end{enumerate}
}
\end{rem}

Turning to issues of memory, we recall that $F$ is transformed in-place to $F^\Sigma$. Therefore, the only additional overhead is the \Que structure. The cost of storing $F$ itself is $\text{O}(npd^2)$: there are $n$ elements in the underlying poset $X$; for each $x \in X$ there are at most $p$ elements $y \in X$ satisfying $x \prec y$, and for each of these we must store at most a $d \times d$ matrix $F_{xy}$. Moreover, since \Que itself may get as large as $n$ for each element of $X$, our worst-case space complexity evaluates to $\text{O}(n^2pd^2)$.

\section{Applications to Distributed Cohomology Computation}\label{sec:apps}

The ability to efficiently compute cellular sheaf cohomology will have implications in those emerging applications [signal processing, sampling, tracking, network coding, optimization, etc.] described in \S\ref{sec:intro}. Given the focus of this paper (on computational cohomology), we do not detail such applications. Instead, we demonstrate an application of sheaf cohomology to the more ubiquitous problem of computing ordinary cohomology over a field. Passing from this to a richer coefficient system can and does facilitate a tremendous simplification of the underlying topological space without loss of cohomology.

There are at least two classical examples of this principle in action: the \v{C}ech approach and the Leray approach. We describe these classical computational methods below, then present a sheaf-theoretic unification. This has the effect of giving a unified interpretation of persistent \cite{zomorodian:carlsson:05,carlsson:09}, zig-zag cohomology \cite{pyramid}, and the Mayer-Vietoris blowup \cite{segal1968classifying} --- all important recent tools in computational topology.

\begin{rem}
{\em
Throughout the remainder of this section, we assume:
	\begin{enumerate}
	\item all topological spaces are compact, Hausdorff and locally contractible;
	\item all covers consist of finitely many open subsets; and
	\item the coefficient ring $\bR$ is a field.
	\end{enumerate}
}
\end{rem}

\subsection{The \v{C}ech Approach}

The following classical approach \cite{alexandroff28} provides a convenient and ubiquitous combinatorial model for representing unwieldy topological spaces.
\begin{defn}
{\em
Let $\cU$ be a cover of a topological space $X$. Its {\em nerve} $N_\cU$ is the abstract simplicial complex whose $n$-dimensional simplices are collections $\sigma = (U_0,\ldots,U_n)$ of cover elements with non-empty {\em support} $\bigcap_0^n U_j$.
}
\end{defn}
Following the usual conflation of an abstract simplicial complex with its cell complex (``geometric'') realization, one has the following (simplified version) of the classical theorem of Leray:

\begin{thm}[Nerve Theorem \cite{leray45, borsuk48}]
Given a topological space $X$ and a cover $\cU$, if the support $U_\sigma \subset X$ of each $\sigma \in N_\cU$ is acyclic (i.e., the reduced cohomology $\tilde{H}^\bullet(U_\sigma;\bR) = 0$ vanishes),
then $H^\bullet(N_\cU;\bR)\cong H^\bullet(X;\bR)$.
\end{thm}

Typically, the cost of guaranteeing acyclicity of supports is that one has to refine substantially the cover $\cU$ and hence greatly increase the number of simplices in $N_\cU$. The following notion is naturally motivated by the desire to compute cohomology with coarser covers and hence fewer simplices.

\begin{defn}
{\em
The {\em \v{C}ech cellular sheaves} $\Cech^n$ associated to the cover $\cU$ of a space $X$ are defined on the nerve $N_\cU$ by the following data. Each $\sigma \in N_\cU$ is assigned the $\bR$-module $\Cech^n(\sigma) = H^n(U_\sigma;\bR)$ and each face relation $\sigma \subset \tau$ is assigned the linear map $\Cech^n_{\sigma\tau}: H^n(U_{\sigma};\bR)\to H^n(U_{\tau};\bR)$ arising from the inclusion of supports $U_\tau \hookrightarrow U_\sigma$.
}
\end{defn}

If all simplex supports are acyclic, then $\Cech^0$ reduces to the constant sheaf on $N_\cU$ and all other $\Cech^n$s are trivial; in the absence of acyclicity assumptions, the following result yields a simple correction.

\begin{prop}
\label{prop:cechsheaf}
{
	Let $X$ be a topological space and $\cU$ a cover whose nerve $N_\cU$ is at most one-dimensional. Then, 	for each $n \in \N$,
\begin{equation}
		H^n(X;\bR) \cong H^0(N_\cU;\Cech^n)\oplus H^1(N_\cU;\Cech^{n-1}).
\end{equation}
}
\end{prop}

We defer the proof to the next section where a more general result is established, but remark here that similar results have been obtained before \cite{pyramid, burghelea2011topological} in the context of zig-zag persistent homology. The central difference between these results and ours is that the existing results depend on the direct-sum decomposition of zig-zag persistence modules into indecomposable modules (or barcodes). On the other hand, our result makes the recognition that these modules are conceived as sheaves over a linear nerve, and moreover that the cohomology of these sheaves can be quickly computed using discrete Morse theory.

\subsection{The Leray Approach}
One can try to compute the cohomology of $X$ with $\bR$ coefficients from a sufficiently nice map $f:X \to Y$ into some simpler space $Y$. If the image of $f$ comes equipped with a cover $\cV$ having nerve $N_\cV$, one can try to pull-back the associated \v{C}ech sheaf on $N_\cV$ along $f$ to yield local information about $X$.

\begin{defn}
{\em
The {\em Leray cellular sheaves} $\Leray^n$ associated to a map $f:X \to Y$ and a cover $\cV$ of $f(X) \subset Y$ are defined over the nerve $N_\cV$ as follows. Each simplex $\sigma \in N_\cV$ is assigned the cohomology of the preimage of its support, i.e., $\Leray^n(\sigma) = H^n(f^{-1}(V_\sigma);\bR)$; furthermore, each face relation $\sigma \subset \tau$ is assigned the map induced on cohomology by the inclusion $f^{-1}(V_\tau) \hookrightarrow f^{-1}(V_\sigma)$.
}
\end{defn}

In the special case where $X = Y$ and $f$ is the identity map, the Leray sheaves clearly coincide with the \v{C}ech sheaves associated to the cover $\cV$ of $X$. Thus, the following result generalizes Proposition \ref{prop:cechsheaf}.

\begin{thm}
\label{thm:leraysheaf}
	Let $f:X \to Y$ be continuous. Assume a cover $\cV$ of the image $f(X) \subset Y$ whose nerve $N_\cV$ is at most one-dimensional. Then, for each $n \in \N$,
\begin{equation}
		H^{n}(X;\bR) \cong H^0(N_\cV;\Leray^n) \oplus H^1(N_\cV;\Leray^{n-1}) .
\end{equation}
\end{thm}

\begin{proof}
		The theorem is a simple consequence of the Leray spectral sequence which packages the cohomology of $X$ into a coefficient system over the space $Y$ from a map $f:X\to Y$ \cite{mccleary}. The restriction to a one-dimensional nerve forces the spectral sequence to collapse on the second page and hence yield the desired isomorphisms. More precisely, for each open $V\subset f(X)$, let $C^n(V;\bR)$ denote the $\bR$-module freely generated by the set of all cochains defined on $V$. Clearly if $V \subset U$, then there is a surjection $C^n(U;\bR) \to C^n(V;\bR)$ defined by restriction of cochains. The sheaf $\cF$ associated to this presheaf of singular cochains is consequently {\em flabby} (see \cite[p.~97]{ramanan2005global}).
		
Consider the following double complex of $\bR$-modules:
	\[
	 \xymatrix{ \vdots & \vdots & \vdots & \vdots\\
	C^2(X) \ar[r] \ar[u] & \bigoplus_{\dim \sigma = 0} \cF^2(f^{-1}(V_\sigma)) \ar[r] \ar[u] & \bigoplus_{{\dim \tau = 1}} \cF^2(f^{-1}(V_\tau)) \ar[r] \ar[u] & 0 \\
	C^1(X) \ar[r] \ar[u] &\bigoplus_{\dim \sigma = 0} \cF^1(f^{-1}(V_\sigma)) \ar[r] \ar[u] & \bigoplus_{\dim \tau = 1} \cF^1(f^{-1}(V_\tau)) \ar[r] \ar[u]  & 0 \\
	C^0(X) \ar[r] \ar[u] & \bigoplus_{\dim \sigma = 0} \cF^0(f^{-1}(V_\sigma)) \ar[r] \ar[u] & \bigoplus_{\dim \tau = 1} \cF^0(f^{-1}(V_\tau)) \ar[r] \ar[u] & 0 }
	\]
	It follows from standard results \cite[Thm~II.5.5, Thm~III.4.13]{bredon1997sheaf} that the rows are exact. By the acyclic assembly lemma \cite{weibel1995introduction}, the spectral sequence converges to the cohomology of the leftmost column, i.e., $H^\bullet(X;\bR)$. If one takes cohomology in the vertical direction, one obtains the defined cochain groups associated to the Leray cellular sheaves $\Leray^n$:
	\[
	 \xymatrix{  \vdots & \vdots & \vdots\\
	\bigoplus_{\dim \sigma = 0} H^2(f^{-1}(V_\sigma)) \ar[r] & \bigoplus_{\dim \tau = 1} H^2(f^{-1}(V_\tau)) \ar[r] & 0 \\
	\bigoplus_{\dim \sigma = 0} H^1(f^{-1}(V_\sigma)) \ar[r] & \bigoplus_{\dim \tau = 1} H^1(f^{-1}(V_\tau)) \ar[r]  & 0 \\
	\bigoplus_{\dim \sigma = 0} H^0(f^{-1}(V_\sigma)) \ar[r] & \bigoplus_{\dim \tau = 1} H^0(f^{-1}(V_\tau)) \ar[r] & 0}
	\]
Taking cohomology horizontally corresponds precisely to computing separately (in parallel, if one wishes) the cohomology of the Leray sheaves $\Leray^n$ over $N_\cV$, thus producing the final stable page of the spectral sequence.
	\[
	 \xymatrix{  \vdots & \vdots & \vdots\\
	 H^0(N_\cV; \Leray^2) & H^1(N_\cV;\Leray^2) & 0 \\
	 H^0(N_\cV; \Leray^1) \ar[rru] & H^1(N_\cV;\Leray^1)  & 0 \\
	 H^0(N_\cV; \Leray^0) \ar[rru] & H^1(N_\cV;\Leray^0) & 0}
	\]
Over a general ring $\bR$, these terms prescribe a filtration of the cohomology, giving rise to extension problems; however, over a field one can read off the cohomology directly.
\end{proof}

Note that the proof indicates precisely where we require the one-dimensional nerve restriction. Without this assumption in place, the second page of the spectral sequence may not be stable and the conclusion of the theorem need not hold.

\subsection{An Example}

\begin{figure}[h!]
\includegraphics[scale=1.2]{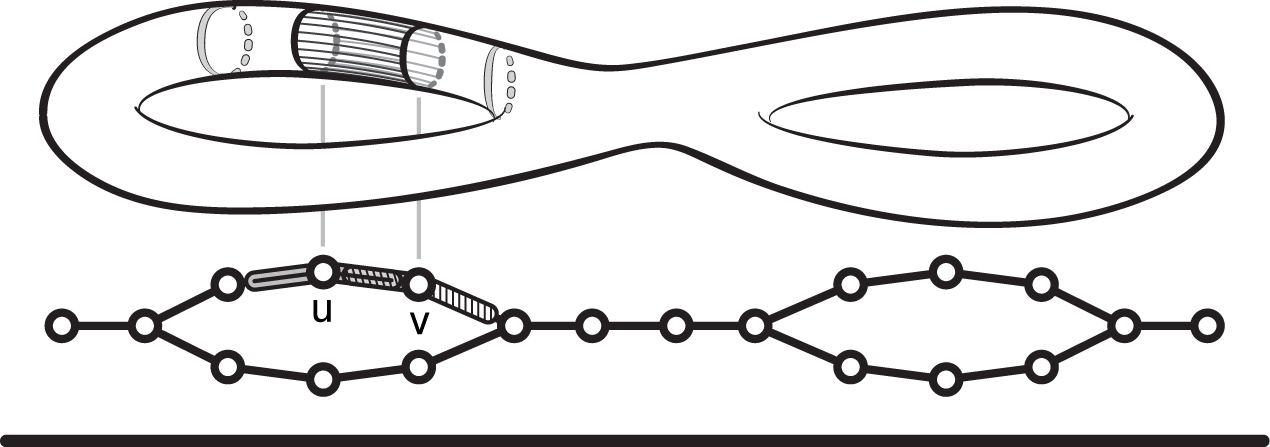}
\caption{A genus-2 surface $X$ hovers over (a subdivision of) its Reeb graph $\Gamma$ associated with downward projection to the real line $\R$. The fibers (lying in $X$) over nodes $u$ and $v$ of $\Gamma$ are highlighted, and their intersection comprises the fiber over the edge $uv$.}
\label{fig:reeb}
\end{figure}

A fairly natural situation where computing cohomology via Theorem \ref{thm:leraysheaf} is advantageous over the obvious alternatives arises when dealing with {\bf Reeb graphs}. Consider a topological space $X$ equipped with a function $f:X \to \mathbb{R}$, and recall that the Reeb graph of the pair $(X,f)$ is a quotient of $X$ by the equivalence relation which identifies two points whenever they lie in the same connected component of $f^{-1}(c)$ for some $c \in \mathbb{R}$.

Let $X$ be a finite CW complex, and consider a continuous function $f:X \to \mathbb{R}$. Given the Reeb graph $\Gamma $ of $(X,f)$ -- for instance, the one illustrated in Figure \ref{fig:reeb} -- one can immediately transform the problem of computing $H^\bullet(X;\bR)$ to that of computing $H^\bullet(\Gamma;\Leray)$, where $\Leray$ is the Leray cellular sheaf on a suitable subdivision of $\Gamma$ associated to the canonical projection $P:X \to \Gamma$. In particular, Theorem \ref{thm:leraysheaf} asserts an isomorphism
\[
H^n(X;\bR) \cong H^0(\Gamma;\Leray^n) \oplus H^1(\Gamma;\Leray^{n-1}),
\]
and in cases where $P$ distributes the cells of $X$ almost evenly over those of $\Gamma$, it is computationally prudent to evaluate the right side in order to determine the left. In order to estimate the advantage, we employ the following complexity parameters:
\begin{enumerate}
\item $N$ is the number of cells in $X$,
\item $d$ is the dimension of (the maximal cells in) $X$,
\item $g$ is the number of cells (vertices and edges) in $\Gamma$, and
\item $K \leq N$ bounds the number of cells in $P^{-1}(v) \subset X$ across vertices $v \in \Gamma$.
\end{enumerate}

In addition to the usual cost of computing $H^\bullet(\Gamma;\Leray)$, one must also take into account the burden incurred when extracting the data which determines $\Leray$, i.e., the stalks and restriction maps. To this end, note that the cost of computing a stalk $\Leray^\bullet(v) = H^\bullet(P^{-1}(v);\bR)$ over a vertex $v$ of $\Gamma$ is $O(K^3)$ via Smith diagonalization of a matrix no larger than $K \times K$ in size. Similarly, each stalk $\Leray(e)$ over an edge $e$ and each restriction map $\Leray(v) \to \Leray(e)$ may be evaluated in $O(K^3)$ time since all matrices involved have their sizes bounded above by $K \times K$. More importantly, these local stalk and restriction map computations may be performed in parallel (there are twice as many restriction maps to compute as there are edges in $\Gamma$), and hence the total cost of computing all the $\Leray$ sheaf data is no more than $O(K^3)$.

Turning now to the computation of sheaf cohomology $H^\bullet(\Gamma;\Leray)$, we note that the relevant cochain complex
\[
0 \to \bigoplus_{\dim v = 0} \hspace{-.1in} \Leray(v) \stackrel{\delta}{\to} \bigoplus_{\dim e = 1}  \hspace{-.1in} \Leray(e) \to 0 \to 0 \to \cdots
\]
contains only two interesting cochain groups (parametrized by the vertices and edges of $\Gamma$ respectively) and a single (potentially) nontrivial coboundary map $\delta$ between them.  Here the matrix representation of $\delta$ consists of at most $g \times g$ blocks arising from restriction maps over incidence relations of cells in $\Gamma$. But each such restriction map furnishes a block no larger than $d \times d$ in size -- after all, the domain and codomain of the restrictions are cohomologies of subcomplexes of $X$, and $X$ itself has dimension $d$.  Thus, the matrix representation of $\delta$ has size bounded above by $gd \times gd$. Even in the complete absence of Morse theoretic simplification, one may therefore evaluate $H^\bullet(\Gamma;\Leray)$ in $O(g^3d^3)$ time. Adding the $O(K^3)$ cost of computing $\Leray$ data, we confront a combined complexity of $O(K^3 + g^3d^3)$ for building the Leray sheaf of $P:X \to \Gamma$ in parallel and evaluating its cohomology.

Thus, the sheaf-cohomological method of computing $H^\bullet(X,\bR)$ is much faster than the traditional methods whenever one has $K^3+g^3d^3 \ll N^3$. In particular, this inequality holds when two mild conditions are satisfied by $P$:
\begin{enumerate}
\item $P$ distributes the cells of $X$ evenly across those of $\Gamma$, so $N \approx Kg$, and
\item the fibers of $P$ have small cohomology relative to their size, i.e., $d \ll K$.
\end{enumerate}
With these assumptions in place, it is straightforward to estimate the ratio $r$ of worst-case complexity when using the Leray sheaf of $P$ to that of directly computing $H^\bullet(X;\bR)$. Clearly, we have
\[
r = \frac{K^3+g^3d^3}{N^3} = \frac{K^3}{N^3} + \frac{g^3d^3}{N^3}.
\]
Using $N \approx Kg$ twice, we have
\[
r \approx \frac{1}{g^3} + \frac{d^3}{K^3}.
\]
Since $g$ may be increased by subdivision and since $d \ll K$ by assumption on the fibers, $r \ll 1$ and the sheaf-theoretic approach enjoys a significant speedup.

\subsection{A Unifying Perspective}

There is a more sophisticated version of the nerve described originally by Segal~\cite{segal1968classifying} which is homotopically faithful to the underlying space independent of the particulars of the cover. This notion has been used in recent applications~\cite{zomorodian2008localized} and parallelizations for homology computation~\cite{lewis2012multicore}.

\begin{defn}
{\em
Let $X$ be a topological space equipped with a cover $\cU$ with nerve $N_\cU$. The {\em Mayer Vietoris blowup} $M_\cU$ associated to $\cU$ is a subset of the product $X \times N_\cU$ defined as follows. The pair $(x,s)$ lies in $M_\cU$ if and only if there is some simplex $\sigma \in N_\cU$ for which $x \in U_\sigma$ and $s \in \sigma$.
}
\end{defn}

Being a subset of the product, $M_\cU$ is equipped with natural surjective projection maps
\[
\xymatrix{ ~ & M_\cU \ar[ld]_{\rho_1} \ar[rd]^{\rho_2} & ~\\
X & ~ & N_\cU}
\]
The map $\rho_1$ has contractible fibers: for any $x \in X$, we have $\rho_1^{-1}(x) = \setof{x} \times \sigma_x$ where $\sigma_x$ is the unique simplex of maximal dimension whose support contains $x$. Thus, the Mayer-Vietoris blowup is homotopy-equivalent to $X$ via $\rho_1$ in full generality. On the other hand, it is easy to see that the map $\rho_2$ fails to have contractible fibers precisely when the simplex supports are not contractible. In fact, given $s \in N_\cU$, the fiber $\rho_2^{-1}(s)$ has the homotopy type of the support of $\sigma_s$, which is the unique simplex of maximal dimension whose realization contains $s$. Since cohomology is a homotopy invariant, this leads to the following observation which unifies the \v{C}ech and Leray approaches.

\begin{prop}
The Leray cellular sheaves $\Leray^n$ associated to the map $\rho_2: M_\cU \to N_\cU$, where $N_\cU$ is covered by (small neighborhoods of the topological) simplices $\setof{\sigma}_{\sigma \in N_\cU}$, are isomorphic to the \v{C}ech cellular sheaves $\Cech^n$ associated to the cover $\cU$.
\end{prop}

\begin{rem}
{\em
We conclude with the following remarks.
\begin{enumerate}
\item The commonality between the \v{C}ech and Leray approaches comes as no surprise to anyone sufficiently familiar with spectral sequences (and would have surprised neither \v{C}ech nor Leray).
\item Both strategies are examples of {\em distributed} cohomology computation because in order to determine the sheaf $\Cech^n$ or $\Leray^n$, one only needs to compute cohomology locally: of a non-trivial intersection of covering sets in the former case, or of a small neighborhood of the fiber $f^{-1}(y)$ in the  latter case. In principle, one can assign each local computation to a different processor, compute the appropriate sheaf cohomology over a decidedly nicer space (either $N_\cU$ or $Y$ depending on the circumstances), and aggregate this information to compute the desired cohomology of $X$.
\item By taking the appropriate linear duals and working with cosheaves \cite{curry}, all of our results transform to computations of homology rather than cohomology.
\end{enumerate}
}
\end{rem}

\section*{ACKNOWLEDGEMENT}
This work was supported in part by federal contracts FA9550-12-1-0416, FA9550-09-1-0643, and HQ0034-12-C-0027.


\end{document}